\documentclass[12pt]{amsart}

\usepackage{amssymb}
\usepackage{epic}
\usepackage{mathrsfs}

\usepackage{enumitem}
\usepackage{graphicx, xcolor}
\usepackage{footnote}
\usepackage{pifont}
\usepackage{dsfont}

\usepackage[norelsize]{algorithm2e}

\usepackage{tikz} 
\usetikzlibrary{patterns}

\newtheorem{theorem}{Theorem}

\newtheorem{corollary}[theorem]{Corollary}
\newtheorem{proposition}[theorem]{Proposition}
\newtheorem{lemma}[theorem]{Lemma}
\newtheorem{que}{Problem}
\newtheorem{conjecture}[que]{Conjecture}
\newtheorem{remark}{Remark}
 \newenvironment{definition}[1]{\noindent{\bf Definition (#1).}}{}

\usepackage[mathscr]{eucal}
\usepackage[margin=2.78cm]{geometry}
\colorlet{darkgreen}{green!50!black}

\usepackage[T1]{fontenc}
\newcommand{\cls}[2][]{\ensuremath{{\mathscr{#2}}_{#1}}}

\graphicspath{{./Images/}}

\renewcommand{\geq}{\geqslant}
\renewcommand{\leq}{\leqslant}

\newcommand{\mA}{\ensuremath{\mathscr{A}}}
\newcommand{\mS}{\ensuremath{\mathscr{S}}}
\newcommand{\mC}{\ensuremath{\mathscr{C}}}

\newcommand{\mH}{\ensuremath{\textsc{H}}}

\newcommand{\bbQ}{\ensuremath{\mathbb{Q}}}

\newcommand{\Lang}[1]{\mathcal{L}_{#1}}
\DeclareMathOperator{\argmax}{argmax}

\newcommand{\set}[1]{\ensuremath{\{{#1}\}}}
\newcommand{\walks}{\operatorname{\sf walks}}
\newcommand{\drift}{\operatorname{\sf drift}}
\newcommand{\grammar}{\operatorname{\sf grammar}}

\def\clsNeutral{\ensuremath{\varepsilon}}

\usepackage[colorinlistoftodos,textsize=small]{todonotes}
\newcommand{\ShowTODO}[1]{{#1}}
\renewcommand{\ShowTODO}[1]{}
\newcommand{\TODOB}[2]{\ShowTODO{\todo[inline, linecolor=#1, backgroundcolor=#1!20!white,bordercolor=#1]{#2}}}

\newcommand{\TODOYann}[1]{\TODOB{gray}{{\bf Yann :} \sf #1}}

\usepackage{microtype}
\date{}
\title{
Taming Reluctant Random Walks in the Positive Quadrant}
\author{Jeremie Lumbroso, Marni Mishna and Yann Ponty}
\address{JL: Department of Computer Science, Princeton University,
  Princeton, NJ, USA;
MM: Department of Mathematics, Simon Fraser University, Canada;
YP: CNRS -- LIX and AMIB project, Ecole Polytechnique and Inria Saclay, France}

\begin{document}
\begin{abstract}
   A lattice walk model is said to be reluctant if the defining step
   set has a strong drift towards the boundaries. We describe
   efficient random generation strategies for these walks.
\end{abstract}
\maketitle
\section{Introduction}
Walks on lattices are fundamental combinatorial classes. They appear
in many guises particularly in formal language theory, queuing theory,
and combinatorics as they naturally encode common relations. A typical
lattice path model is a set of walks defined by a fixed, finite set of
allowable moves (called the \emph{step set}), and a region to which
the walks are confined (typically a convex cone).  The exact and
asymptotic enumeration of lattice paths restricted to the first
quadrant (known as quarter plane models) have been a particularly
active area of study of late because of some new and interesting
techniques coming from different areas of computer algebra and complex
analysis~\cite{BoKa09, BoMi10,KuRa11, Rasc12}.

Efficient uniform random generation is useful to study the typical
large scale behavior of walks under different conditions.  Models
which restrict walks to the upper half plane can be specified by an
algebraic combinatorial grammar~\cite{Duch00,BaFl02}. Consequently,
efficient random generation schemes can be obtained using several
systematic strategies, such as recursive generation~\cite{FlZiVa94}
and Boltzmann sampling~\cite{DuFlLoSc02}.

Intriguingly, walks restricted to the first quadrant are more
complex. Rare is the quarter-plane model with an algebraic generating
function that cannot be trivially reformulated as a half-plane model.
Overwhelmingly, the cyclic lemma, combinatorial identities and other
grammar-based techniques that are so fruitful in the half-plane case,
do not easily apply. Furthermore, there is only a small proportion of
models whose generating function satisfies a differential equation
with polynomial coefficients\footnote{For example, amongst the 20~804
  small step models with less than 6 steps in 3 dimensions, only
  around 150 appear to be D-finite~\cite{BoBoKaMe15}.}, again
excluding a potential source of direct, generic generation
techniques~\cite{BaBoJa13}.

Rejection sampling is the term for a general technique where one
generates from a simpler superclass, and then rejects elements until
an element from the desired class is obtained. In the case of lattice
paths, a naive rejection strategy could use unrestricted walks as a
superset. This is only practical for those quarter-plane models whose
counting sequences grow essentially like those of the unrestricted
walks. Such is the case when the \emph{drift\/}, or vector sum of the
stepset, is positive coordinate-wise. Anticipated rejection can also
be used when the drift is $\mathbf{0}$, and provides a provably
efficient algorithm~\cite{BaSp14}.  However, any such strategy is
demonstrably doomed to failure when the drift of the step set is
negative in any coordinate, as the probability of generating long
unconstrained walks which remain in the first quadrant becomes
exponentially small. One strategy in the literature is to change the
probability on the allowable steps, and consequently forgo the
uniformity of the generation~\cite{Bousquet11}. It appears then that
the problem of efficient, uniform random generation algorithms for
generic quarter plane lattice path models is a relatively undeveloped
topic.

\subsection*{Our contribution}The main result of this paper is an
efficient rejection algorithm for the \emph{uniform\/} random
generation of walks in the quarter plane. It is an application of
recent results due Johnson, Mishna and Yeats~\cite{JoMiYeXX}, Garbit
and Raschel~\cite{GaRa14} amongst others. It is provably efficient and
straightforward to implement. It is most impressive on walks whose
drift is negative in both coordinates, a property we call
\emph{reluctant}, but it also offers notable gains for any model which
tends to either boundary.

More precisely, we describe a strategy in which every walk of
length~$n$ is generated with equal likelihood. The efficacy result
holds for quarter plane models with any step set, and is easily
generalized to higher dimensions.  Figure~\ref{fig:bigwalk}
illustrates a walk of over 18000 steps that was generated uniformly at
random for the quarter plane model with reluctant step
set \[\mS=\{(1,0), (0,1), (-1, 0), (1, -1), (-1, -1), (-2, -1) \}.\]
%
\begin{figure}\center
\includegraphics[width=0.75\textwidth]{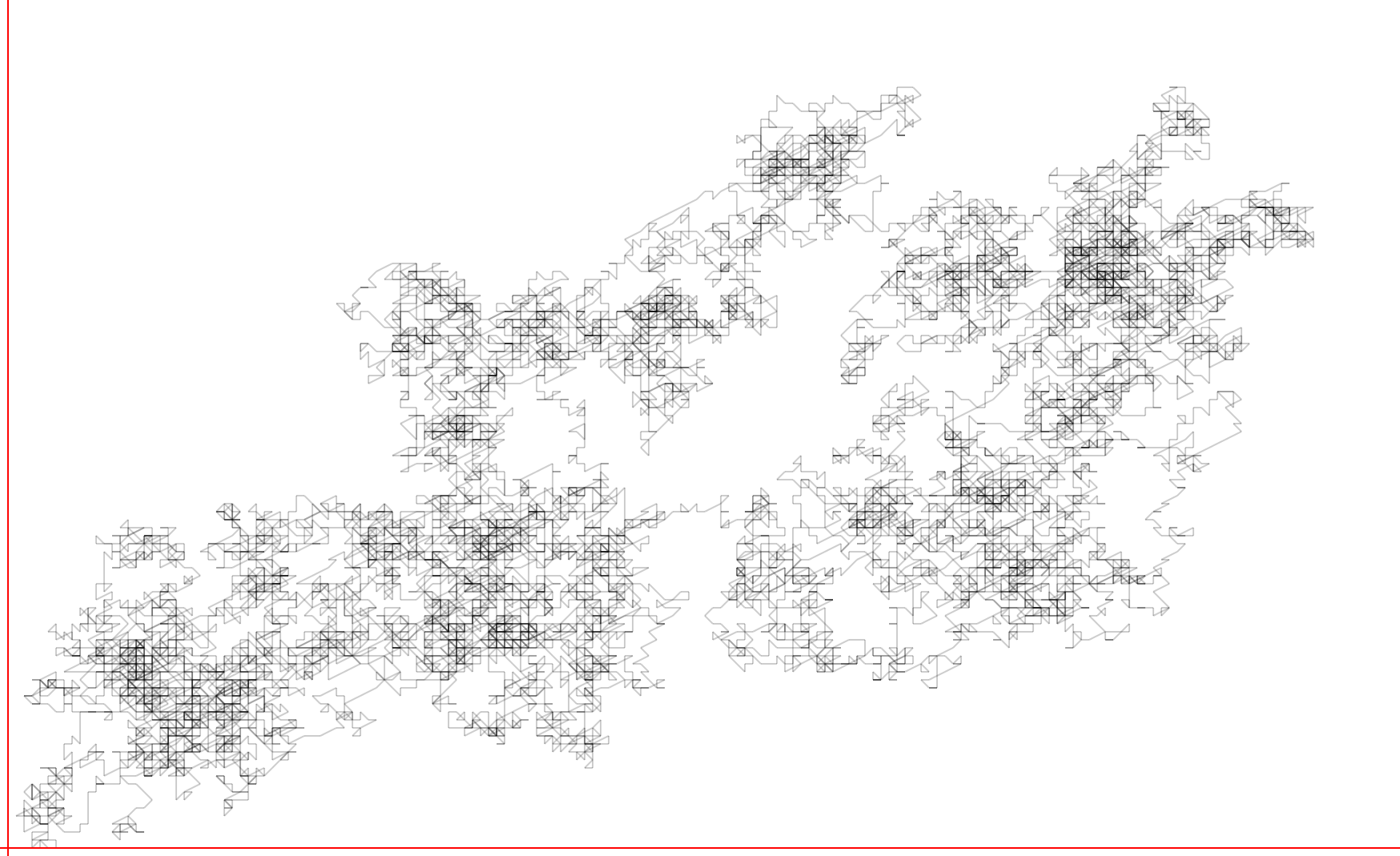}
\caption{\sc A random walk with 18 000 steps in the quarterplane using the 
  stepset \mbox{$\mS=\{(1,0), (0,1), (-1, 0), (1, -1), (-1, -1), (-2, -1) \}$}.}
\label{fig:bigwalk}
\end{figure}
%
The probability of generating a walk of this length by rejection from
the set of unrestricted sequences of steps~$\mS^*$ is less than
$\frac{5.3299}{6}^{18000}\sim 1.75\cdot 10^{-926}$. However,
with our strategy, it was generated (relatively quickly). 

Rejection from an unrestricted walk is not the only competition. For
the purposes of comparison, we describe a recursive strategy which
requires exact enumeration results to be tabulated in advance. This
has potential to be efficient, and is insensitive to the drift of the
model, but does require a lot of storage. We discuss this algorithm in
Section~\ref{sec:Recursive}.

Our alternative sampler is based on a straightforward combinatorial
interpretation of an enumerative result. Roughly, we use that for any
quarter plane model, there is a corresponding half plane model such
that asymptotically, both models have the same exponential growth
factor. This implies that a rejection strategy from this half plane
has sub-exponential rate of rejection.  The sub-exponential factors
are conjectured to also match in many cases, suggesting that it is in
fact a particularly efficient strategy.

Asymptotic enumerative results are recalled in the next section.  A
{\em baseline} algorithm, based on a trivial recurrence, is presented
in Section~\ref{sec:naive}. Section~\ref{sec:rejection} describes our
main rejection algorithm. Its practical implementation depends on the
rationality of the slope of the half-plane model, and is discussed in
Subsection~\ref{sec:grammars}.  We conclude with some remarks
regarding implementation aspects, along with possible extensions.

\section{2D lattice path basics}
A 2D lattice path model is a combinatorial class consisting of walks
on the 2D integer lattice, starting at the origin, and taking steps
from some finite multi-set $\mathcal{S}\subset \mathbb{Z}^2$ of
allowable steps. In this work, we consider the restriction of such
walks to the positive quadrant $Q=\mathbb{Z}^2_{\geq 0}$, although the
general strategy works for a wider set of cones. We use the half-plane
$H_\theta$ defined by a line through the origin:
\[
H_\theta= \{ x \sin \theta + y \cos\theta \geq 0 \}.
\]

For a fixed, finite step set~$\mS\subset\mathbb{Z}^2$, a given cone
$C$, and positive integer $n$ we define $\walks (C, \mS, n)$ to be the
class of walks of length~$n$ starting at the origin, taking steps in
$\mS$, and staying in $C$. Formally,
\[
  \walks (C, \mS, n)=\{ x_0,x_1,\dots,x_n\mid x_0= (0,0) \,\wedge\,
  x_{j+1}-x_j\in\mathcal{S} \,\wedge\, x_i\in C \}.
\]
The complete class is given by $\walks (C,\mS)=\bigcup_{n\geq 0}
\walks (C, \mS, n)$.

We use the following enumerative quantities in our analysis:
\begin{equation}\label{eqn:qn}
q_n^\mS=|\walks (Q,\mS, n)|\quad h^\mS_n(\theta)=|\walks (H_\theta, \mS, n)|.
\end{equation}%
The asymptotic regime for $q_n^\mS$ is always of the form
\begin{equation}q^\mS_n \sim  \gamma\,\rho^{-n}\, n^{-r}, \label{eq:asymptwalks}
\end{equation}
for real numbers $\rho$ and $r$. We refer to $\rho^{-1}$ as the
exponential growth factor of the model.  The asymptotic regime
critically (although not exclusively) depends on the \emph{drift\/} of
the step set~$\mS$, defined as $\drift (\mS)= \sum_{s\in S} s$. %
  A walk model $\walks(Q, \mS)$ is said to be \emph{reluctant\/} when
  $\drift (\mS)=(\delta_1, \delta_2)$ with $\delta_1<0$ and
  $\delta_2<0$.

  Reluctant models for the positive quadrant have exponential growth
  factors that are lower than the number of steps. It follows that the
  naive algorithm that performs rejection from unconstrained walks,
  has exponential time complexity, motivating our algorithmic
  contribution. Indeed, even when one of $\delta_1<0$ or $\delta_2<0$,
  the exponential growth factor can be less than the number of steps.

\section{Basic recursive random generator}\label{sec:naive}
The exact value of~$q^\mS_n$ can be expressed using a recurrence. 
This motivates a straightforward instance of the recursive method~\cite{Wilf1977,FlZiVa94}, where steps are simply drawn sequentially, using probabilities that depend both on the current position reached, and the number of remaining steps. 

\subsection{Exact enumeration of walks}
Define $q^\mS_n(x,y)$ to be the number of positive suffixes of walks
in $\walks (Q, \mS, n)$ which start from the point $(x, y)$ and remain in the 
positive quadrant. Such suffix walks of length $n$ can be factored as a first step $(i,j)\in\mS$, keeping the walk in the positive quadrant, followed by another positive suffix of length $n-1$ starting at $(x+i,y+j)$. This leads to the recurrence:
\begin{equation}\label{eq:qnij}
q^\mS_n(x,y) = \begin{cases} \displaystyle
\sum_{\substack{(i, j)\in \mS \text{ s.t.}\\x+i\ge0,y+j\ge0 }} q^\mS_{n-1}(x+i,y+j)& \mbox{if }n>0,\\
1 & \mbox{if } n=0
\end{cases} 
\end{equation}
A quadrant walk is also the positive suffix of a walk starting at $(0,0)$, thus $q^\mS_n := q^\mS_n(0,0)$. This recurrence can also be trivially adapted to handle general cones, higher dimensions, or for further constraining the end-point, e.g. to count/generate meanders, or walks ending on the diagonal.

\subsection{Algorithm and complexity analysis}
\label{sec:Recursive}
Once the cardinalities $q^\mS_n(x,y)$ are available, a uniform random walk is generated, by choosing one of the steps with probabilities proportional 
to the number of possible suffixes.

\begin{enumerate} 
 \item {\bf Preprocessing.} Precompute $q^\mS_{n'}(x,y)$ for each $n'\in [0,n]$ and $(x,y) \in [0,n\cdot a]\times[0,n\cdot b]$, where $a:= \max_{(i,j)\in\mS}i$ and $b:= \max_{(i,j)\in\mS}j$;\label{step:precomp}
 \item {\bf Generation.} Initially starting from $(0,0)$ and $n':=n$, iterate until $n'=0$:\label{step:elongation}
 \begin{enumerate}
 \item Choose a step $(i,j)\in\mS$ with probability $q^\mS_{n'-1}(x+i,y+j)/q_{n'}(x,y)$; 
 \item Add $(i,j)$ to the walk, update the current point
   ($(x,y) := (x+i, y+j) $), and decrease the remaining length
   ($n':=n'-1$);
\end{enumerate}
\end{enumerate}

\begin{theorem}[Complexity/correctness]
\label{thm:storage}
  The random uniform generation of~$k$ 2-dimensional walks confined
  to the positive quadrant can be performed in $\Theta(k\cdot n +
  n^{3})$ arithmetic operations, using storage for $\Theta(n^{3})$
  numbers.
\end{theorem}

\begin{proof}

  The preprocessing stage should only be computed once in the
  generation of $k$ sequences. It involves
  $\Theta(|\mS| \cdot n^{d+1})$ arithmetic operations, and requires
  storage for $\Theta(n^{d+1})$ large numbers. The generation of a
  single walk requires the generation of $\Theta(n)$ random numbers
  and, for each of them, their comparisons to $\Theta(|\mS|)$ other
  numbers.
  
 An induction argument establishes the correctness of the algorithm. Assume that, for all $n'<N$ and $(x,y) \in [0,n'\cdot a]\times[0,n'\cdot b]$, the positive suffixes are uniformly generated, a fact that can be verified when $n'=0$. Then for $n'=N$, the algorithm chooses a suitable step $(i,j)\in \mS$, and then recursively generates a -- uniform from the induction hypothesis -- suffix from the updated position. The probability of generating any such walk is therefore
 \[ \mathbb{P}(w) = \frac{q^\mS_{N-1}(x+i,y+j)}{q_N(x,y)} \times \frac{1}{q^\mS_{N-1}(x+i,y+j)} = \frac{1}{q^\mS_{N}(x,y)}\]
 and we conclude with the uniformity of the generation.
\end{proof}

In practice however, the memory consumption of the algorithm grows in
$\Theta(n^4)$ bits, which limits the utility of this strategy to
$n<500$. Thus the above algorithm only serves as a baseline for our
alternative based on rejection.

\section{Efficient rejection sampler from 1D models}\label{sec:rejection}
We recall some basics of rejection sampling for our  analysis.
Let~$\mA$ be a combinatorial class which contains the sub-class
$\mC$. Given a random sampler for $\mA$, we can use a rejection
strategy to make a random sampler for $\mC$. Let $a_n$ and~$c_n$
respectively count the number of elements of size $n$ in $\mA$ and
$\mC$.   Following
Devroye~\cite[Chapter II.3]{Devr86}, we say that class $\mA$
\emph{efficiently covers\/} $\mC$ if
\[\left(\frac{a_n}{c_n}\right) \in \mathcal{O}(n^{p}), \]  
Here  $p \ge 0$ is some constant independent of $n$. In other words, 
asymptotically, the expected number of elements drawn from~$\mA$
before generating an element in~$\mC$ is polynomial in~$n$. Ideally~$p$ is as
small as possible.

\subsection{Candidate superclass: Half-plane model}
Our algorithm arises from the surprising observation made
by Johnson, Mishna, and Yeats~\cite{JoMiYeXX}, later proven by Garbit
and Raschel~\cite{GaRa14}: 
\begin{theorem}[Garbit and Raschel~\cite{GaRa14}]\label{th:GR2014}
  Consider a step set $\mS$, let
  $\rho(\theta)^{-1}:=\lim_{n\rightarrow \infty}
  h^\mS_n(\theta)^{1/n}$ be the exponential growth factor of the
  half-plane model $\walks (H_\theta, \mS)$, and define
  \[\theta^* := \argmax_{0\leq \theta \leq \pi/2} \rho(\theta),\]
  Then the growth factor $\rho^{-1}:=\lim_{n\rightarrow \infty}
  (q^\mS_n)^{1/n}$ of walks in the positive quadrant $Q$ satisfies:
\begin{equation}
\label{eq:qnhn}
\rho=\rho(\theta^*).
\end{equation}
\end{theorem}
This says that the exponential growth of the quarter-plane
model is equal to the exponential growth of a superclass 
half-plane model. Furthermore the value of $\theta^*$ is explicitly computable.
\begin{corollary}\label{thm:EfficientCover}
  The combinatorial class $\walks (H_{\theta^*}, \mS)$ efficiently covers
  $\walks (Q,\mS)$.
\end{corollary}
Next we consider the sub-exponential factors, as this gives the
polynomial complexity of the rejection.  On the side of the half-plane
walks, the asymptotic formulas for~$h^\mS_n(\theta)$ can be deduced 
from the complete generating function study of
Banderier and Flajolet~\cite{BaFl02}. The sub-exponential factors are
either $n^0, n^{-1/2}$, or $n^{-3/2}$, depending on the drift of the model
(positive, zero and negative respectively).
 
For quarter-plane walks, the picture is less complete. The case of
excursions for models with zero drift was described by Denisov and
Wachtel~\cite{DeWa15}, and from this work
Duraj~\cite[Theorem~II]{Dura14} was able to conclude explicit formulas
for reluctant walks: Let $S(x,y)=\sum_{(i,j)\in \mS} x^iy^j$, and let
$(\alpha, \beta)$ be the unique positive critical point of
$S(x,y)$. Such a point always exists, provided that~$\mS$ satisfies
some non-triviality conditions. Then, \begin{equation}\label{eq:qnsim}
  q^\mS_n \sim \gamma\,\rho^{-n}\, n^{-r},
\end{equation}
where $\rho$
and $r$ satisfy
\[
\rho=\frac{1}{S(\alpha,\beta)}\text{ and }
r=1+\pi\arccos
\frac{S_{xy}(\alpha, \beta)}{\sqrt{S_{xx}(\alpha, \beta)    S_{yy}(\alpha, \beta)}}.
\]
 \SetKwFunction{UniformDraw}{UniformDraw}
 \SetKwFunction{Map}{2DMap}
\begin{algorithm}[t]
 \KwData{Reluctant step set $\mS\subset\mathbb{Z}^2$, length $n$}
 \KwResult{$w\in \walks(Q, \mS, n)$ drawn uniformly at random}
 \tcp{Determine optimal slope $m=arctan(\theta^*)$ following \cite{JoMiYeXX}}
 \lIf{$\mS$ is singular}%
 {$m\gets 0$  }%
 \Else{
  Set $S(x,y)=\sum_{(i,j)\in \mS} x^iy^j$\;
  Determine  $(x,y)=(\alpha, \beta)$, the unique positive solution of
    $\frac{d}{dx}S(x,y)=\frac{d}{dy} S(x,y)=0$\;
 \lIf{$\beta=1$}%
 {$m=\infty $}%
 \lElse{$m=\ln \alpha/\ln\beta$}%
 }%
 \tcp{Create suitable grammar $\mathcal{G}$}
\lIf{ $m=\infty$}{$p\gets 1$ and $q\gets 0$}
\lElseIf{ $m$ is rational}{find $p,q\in\mathbb{N}$ so that $m=p/q$}
\lElse{find $p/q$,  a $1/\sqrt{n}$-rational approximation to $m$}
{$\mA\to\{ip+jq:(i,j)\in \mS\}$\;
$\mathcal{G}\to\grammar (\mA)$\;}
 \tcp{Main rejection loop}
 \Repeat{$\Map{$w$} \in Q$}{$w \to \UniformDraw{$\mathcal{G},n$}$}

 \caption{Outline of our rejection algorithm. \protect\UniformDraw{$\mathcal{G},n$} denotes a uniform sampler of walks of length $n$ for the grammar $\mathcal{G}$, and \protect\Map{$w$} indicates the reintepretation of $w$ as a sequence of 2D steps.\label{alg:fullmonty}}
\end{algorithm}

\subsection{The algorithm}
Algorithm~\ref{alg:fullmonty} implements a classic rejection from a
carefully-chosen half-plane model.

\begin{theorem}[Complexity of Algorithm~\ref{alg:fullmonty}]
  Let $\mS$ be a reluctant walk model and let $M(n)$ denote the time
  complexity of generating a walk for the half-plane model
  $H_{\theta^*}$.  The expected time taken by Algorithm~\ref{alg:fullmonty} to
  generate a walk in the positive quadrant is in $\Theta\left(M(n)\times
    n^{r-3/2}\right).$
\end{theorem}
This immediately follows from formula~\eqref{eq:qnsim}, from which we
deduce that the expected number of trials is
${h^\mS_n(\theta)}/{q^\mS_n} \in \Theta(n^{r-3/2}).$
For reluctant small step models, one has $3.3< r < 7.5$.  More
recently, Garbit and Raschel have conjectured formulas for the
sub-exponential factor in the general case, and
remarkably suggest that for many (non-reluctant) models,
${h^\mS_n(\theta)}/{q^\mS_n} \in \mathcal{O}(1)$.
 
Next we address the efficient uniform random generation of walks in
$\walks (H_{\theta^*}, \mS)$.

\subsection{Half plane models as unidimensional walks}
We now describe efficient samplers for half-plane models $\walks
(H_\theta, \mS)$.  Remark that walks in any half-plane can be
generated as positive 1D walks, by taking 1D steps that are the
orthogonal projections of those in $\mS$ onto the half-plane boundary.

A unidimensional model is defined by a set $\mA\subset\mathbb{R}$. The
nontriviality conditions imply that $\mA$ contains both a positive and
a negative element.  The associated class of walks begins at $0$,
takes steps which are elements from $\mA$ such that the sum over any
prefix of the walk is nonnegative. If $\mA$ is a multiple of a set of
integers, then the class is modelled by a context free grammar, which
we describe in the next section. Otherwise, the class cannot be
trivially modelled by a context-free grammar, as is proven in
Section~\ref{sec:notCFG}.

Given $\mS$ and $\theta$, we define the associated unidimensional
model \[\mA(\theta)=\{i\sin\theta + j\cos\theta: (i,j)\in\mS\}.\]
The classes $\walks(H_\theta,\mS)$ and
$\walks (\mA(\theta), \mathbb{R}_{\geq 0})$ are in a straightforward
bijection. 
%

\begin{remark}[\cite{JoMiYeXX}]\label{rem:misc}
  If $\mS$ defines a non-trivial 2D quarterplane model, then
  $\mA(\theta)$ defines a non-trivial unidimensional model.  Moreover
  if $\mS$ is relunctant, then the drift of $\mA(\theta)$ is negative.
  Finally multiplying steps by a positive constant does not affect 
  the language of positive walks.
\end{remark}

Two cases arise, depending on whether or not $\mA(\theta^*)$ consists
of rational-valued steps (up to rescaling). This is equivalent to
asking if the slope of the boundary of the half-plane,
$m=\tan(\theta^*)$ is rational.
\subsection{Case 1: Rational projected steps}
\label{sec:grammars}
When $\tan(\theta^*)$ is rational, the steps in $\mA(\theta^*)$ can be
scaled to be integers, as mentioned in Remark~\ref{rem:misc}, therefore 
we consider unidimensional models $\mA$ which consist of integer-valued steps.

Combinatorial specifications and, specifically, context-free grammars
can then be used for random generation.  Context-free grammars are
indeed suitable to describe objects following rules which depend on a
single, integer counter---and place certain, finite constraints on
this counter. For the purpose of random walks, this counter may
typically keep track of the height of the walk, and be constrained to
always remain positive (i.e., the walk remains above the
$x$-axis). From a grammar, random objects can be sampled using a
variety of generic methods. More generally, this is equivalent to
saying that grammars can describe walks that are confined within a
\emph{half-plane}.

%

To build the grammar $\grammar(\mA)$ for a unidimensional model
defined by step set $\mA$, we first distinguish the positive, negative
and neutral steps
\begin{align*}
  \mA^+ &:= \set{ a \ | \ a\in \mA \text{ and } w(a) > 0 },&
  \mA^-&:= \set{ a\ | \ a\in \mA \text{ and } w(a) < 0 },
\end{align*}
\begin{equation*}
  \mA^0 := \set{a \ | \ a\in \mA \text{ and } w(a) = 0 }\text{,}
\end{equation*}
and define the largest upward and downward step lengths
\begin{align*}
  \bar{a} &:= \max \mA^+ &  \bar{b} &:= -\min \mA^-.\end{align*}
Note that both of these lengths are positive, and are well-defined
when  the step set satisfies the conditions of non-triviality.

Using these three sets, and these two values we define the associated
grammar $\grammar(\mA)$, whose terminals are given by $\mA$, and
non-terminals are defined as follows:
\def\clsPaux{\cls[{\mathrm{aux}}]{P}}%
\begin{align*}
  \cls{P}   &= \cls{D} \times \clsPaux&
  \cls[i]{L}&= \sum_{a\in \mA\atop{w(a) = i}} a \; +
  \sum_{k=i+1}^{\min(\bar{a}, i+\bar{b})} \cls[k]{L} \cls[k-i]{R}\\
  \clsPaux &= \clsNeutral +
                          \sum_{k=1}^{\bar{a}} \cls[k]{L} \times \clsPaux &
                            \cls[j]{R}&= \sum_{b\in \mA\atop{w(b) = -j}} b \; +
  \sum_{k=j+1}^{\min(j+\bar{a}, \bar{b})} \cls[k-j]{L} \cls[k]{R}
\end{align*}
\begin{equation*}
  \cls{D}   = \sum_{c\in S\atop{w(c) = 0}} c \times \cls{D} \; +
               \sum_{k=1}^{\max(\bar{a},\bar{b})} \cls[k]{L} \times \cls{D}\times\cls[k]{R}\times\cls{D}
\end{equation*}
This follows from Duchon~\cite{Duch00},
Bousquet-M\'elou~and~Ponty~\cite{BoPo08}, with minor corrections to
the indices that prevent the grammar from referencing undefined
rules. The decomposition of a walk is unique and a schematic of a
typical decompostion is presented in Figure~\ref{fig:decomp}.
\begin{figure}\center
\begin{tikzpicture}[scale=1.4]
   \draw [thick,domain=0:170, <-, red] plot ({cos(\x)}, {sin(\x)});
   \draw [thick,domain=0:180, <-] plot ({cos(\x)+2}, {sin(\x)});
\draw [thick,domain=10:180, <-, blue] plot ({cos(\x)+4}, {sin(\x)});
\draw [dashed, thick, color=gray](-1,0) -- (5,0);
\draw [dashed, thick, color=gray](-1.1,-.5) -- (7,-.5);
\draw[ultra thick, ->, red] (-1.1,-.5) -- (-0.985,0.174) node[anchor=west] {$a$};
\draw[ultra thick, ->, blue] (4.985,0.174) -- (5.1,-.5) node[anchor=east] {$b$};
\draw [thick,domain=0:180, <-] plot ({cos(\x)+6.1}, {sin(\x)-0.5});
\draw[thick, <->, gray](2,-0.5)--(2,0)node[anchor=east]{};
\draw (1.75, -0.25)node{$k$};
\draw (2, 0.5)node{$\mathcal{D}$};
\draw (6.1, 0)node{$\mathcal{D}$};
\draw [red](0, 0.5)node{$\mathcal{L}_k$};
\draw [blue](4, 0.5)node{$\mathcal{R}_k$};
\end{tikzpicture}
\caption{Typical decomposition of a walk  in $\mathcal{D}$ with first
  step of height of height $a\geq k$.}
\label{fig:decomp}
\end{figure}
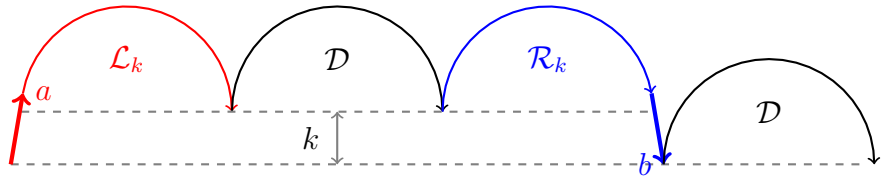


Given the step set $\mA$, $\mathcal{G}=\grammar (\mA)$ can be built in
constant time (proportional to $\max(\bar{a}, \bar{b})^2$). To
generate an element from a context free grammar, one either uses
recursive methods~\cite{Wilf1977} or Boltzmann
generation~\cite{DuFlLoSc02}.  The grammar here is straightforward, so
most common optimizations apply~\cite{Goldwurm1995}.

\begin{theorem}[Complexity of rational half-plane sampling]
Let $\walks (\mathbb{R}_{\geq 0}, \mA)$ be a non-trivial unidimensional model defined by a rational multiset~$\mA\subset\mathbb{Z}$. The uniform random generation of $k$ walks of length~$n$ in
$\walks (\mathbb{R}_{\geq 0}, \mA)$ can be performed in $\mathcal{O}(k\cdot n\log n)$ arithmetic operations using storage for $O(1)$ numbers. 
\end{theorem}

\begin{corollary}
  When the step set $\mS$ yields a rational $\mA$ unidimensional
  projection, Algorithm~\ref{alg:fullmonty} generates $k$ walks in the
  positive quadrant using $\mathcal{O}(k\cdot n^{r-1/2}\log n)$
  arithmetic operations, where $r$ is the exponent of the
  subexponential term in the asymptotics of $\walks (Q,\mS)$ .
\end{corollary}

\subsection{Case 2: Non-rational  projected steps}
\label{sec:irrationalslope}
When the projected step set $\mA(\theta^*)$ contains non-rational steps, then the associated language is not context-free, and grammars can no longer be used directly. However, it is still possible to use a rational approximation of the perfect half-plane model, at the expense of the algorithmic efficiency.

\subsubsection{Contextuality of associated languages}\label{sec:notCFG}
\begin{lemma}
  Let $\mS\subset\mathbb{Z}^2$ be a finite set which defines a
  non-trivial quarterplane model. Let $\theta^*$ be angle determined
  by Theorem~\ref{th:GR2014} and assume furthermore that
  $m=\tan(\theta^*)$ is irrational.  Then, the language $\Lang{m}$
  whose alphabet is made from the pairs $(i,j)\in \mS$, and the words
  are restricted to walks in $\walks(H_{\theta^*}, \mS)$ is not
  context-free.
\end{lemma}

\begin{proof}Consider two steps $a, b\in \mA$, encoded by symbols $s_a$ and $s_b$, such that $a>0$ and $b<0$ and $a/b$ is irrational. The existence of such steps follows from the non-triviality of $\mS$, and the irrationality of $\tan(\theta^*)$.
  First, recall that the intersection of a context-free language
  and a rational language is a context-free language. If the
  intersection language
$$\Lang{m}^{\cap}=\{s_a^*s_b^*\}\cap \Lang{m} = \{s_a^is_b^j\mid a\cdot i-b\cdot j\ge 0\}$$ is not context-free, then neither is $\Lang{m}$. 

The fact that $\Lang{m}^\cap$ is not context free can be proven
using the context-free version of the pumping lemma, which states
that, if $\Lang{m}^\cap$ is context-free, then there exists a word
length $p$ above which each word $w\in\Lang{m}^\cap$ can be decomposed
as $w=x.u.y.v.z$ such that $|u.y.v|\le p$, $|u.v|\ge 1$, and
$\{x.u^i.y.v^i.z\mid i\in\mathbb{N}\}\subset\Lang{m}^\cap$.

Let $\Delta(w) = |w|_{s_a}\cdot a - |w|_{s_b}\cdot b$ denote the (signed) final distance to the half plane, we establish the following technical lemma. 
\begin{lemma}For any $p\ge 0$, there exists a word $w^*\in\Lang{m}^\cap, $ $|w^*|>p$, such that $\Delta(w^*)<\Delta(w)$ for all $w\in\Lang{m}^\cap$, $|w|<|w^*|$.\label{lem:existsCloserWord}
\begin{proof}
Assume that $p$ is given, and let $\Delta^{\le p}$ denote the smallest distance to the half plane of a word of length $\le p$, reached by some word $s_a^{x^{\bullet}}s_b^{y^{\bullet}} \in \Lang{m}^\cap$ of length $x^\bullet+y^\bullet\le p$. 

First we constructively show the existence of a word of length greater than $p$, whose final distance to the half-plane is smaller than $\Delta^{\le p}$.
Consider the word $$w^{\circ}:=s_a^{K\cdot x_\bullet}s_b^{K\cdot y_\bullet} s_b\quad\text{ where }\quad K:=\left\lceil\frac{b}{\Delta^{\le p}}\right\rceil.$$
Since both the slope and ratio $a/b$ are irrational, then $\Delta^{\le p}\neq 0$ and such a word exists.
The final distance to the half plane of $w^{\circ}$ is given by:
\begin{align*}  
  \Delta(w^\circ) 
&=  K\cdot\Delta^{\le p} - b = \left(\left\lceil\frac{b}{\Delta^{\le p}}\right\rceil-\frac{b}{\Delta^{\le p}}\right)\cdot \Delta^{\le p}< \Delta^{\le p}.
   \end{align*}

Consider now the smallest word $w^*\in\Lang{m}^\cap$ such that $\Delta(w^*)<\Delta^{\le p}$.
Such a word exists since $\Delta(w^\circ)<\Delta^{\le p}$ and clearly obeys $|w^*|>p$. Since $w^*$ is the smallest word such that $|w^*|>p$ and $\Delta(w^*)<\Delta^{\le p}$, then one has $\Delta(w^*)<\Delta^{\le |w^*|}$ which proves our claim.
\end{proof}
\end{lemma}

Let us now investigate the possible factorizations as $x.u.y.v.z$ of the word $w^*$, whose existence is established by Lemma~\ref{lem:existsCloserWord}, and show that neither of them satisfies the pumping lemma. Focusing on $u$ and $v$, remark that neither of them should simultaneously feature both kinds of steps, otherwise any word $w^{[1]}=x.u^2.y.v^2.z\notin\Lang{m}^\cap$, as it would feature at least two peaks. It follows that any satisfactory decomposition must be of the form $u=s_a^i$ and $v=s_b^j$, and the non-rationality of $a/b$ implies that $\Delta(u.v)\neq 0$. If $\Delta(u.v)<0$, then the word $w^{[2]}=x.u^r.y.v^r.z$, $r>\lceil \Delta(w^*)/\Delta(u.v)\rceil$ is such that $\Delta(w^{[2]})<0$, and therefore $w^{[2]}\notin \Lang{m}^\cap$. 
If $\Delta(u.v)>0$, then let us observe that $u.v\in\Lang{m}^\cap$ and has total length $i+j<p$, therefore Lemma~\ref{lem:existsCloserWord} implies that $\Delta(u.v)>\Delta^{|w'|}$. It follows that the word $w^{[3]}=x.u^0.y.v^0.z$ has final distance to the slope $\Delta^{|w'|}-\Delta(u.v)<0$ and therefore $w^{[3]}\notin\Lang{m}^\cap$. Having found all possible decompositions lacking in some respect, we conclude that $\Lang{m}^\cap$ is not a context-free language, and neither is $\Lang{m}$.
\end{proof}

\subsection{Rational approximations}
\label{sec:ratapprox}
All is not lost in the case of an irrational slope model, however, as
we can define an approximation to the slope that is 
sufficiently close to the optimal slope to ensure polynomial-time rejection. 

\begin{definition}{$\delta$-rational approximation}
  A half-plane model $\mH_{\theta_r}(\mS)$ is a $\delta$-rational
  approximation of a half-plane model $\mH_\theta(\mS)$ if and
  only if $m_r := \tan{\theta_r} \in \bbQ$ and $|\tan{\theta}-\tan{\theta_r}|\le \delta$.
\end{definition}
\begin{proposition}\label{thm:approx}
  For any model $\walks(H_\theta, \mS)$ and $\delta>0$ a desired precision, there exists a
  grammar with $\mathcal{O}(1/\delta)$ non-terminals and
  $\mathcal{O}(1/\delta^2)$ rules, which generates a $\delta$-rational
  approximation $\walks(\mH_{\theta_r},\mS)$ of $\walks(H_\theta,\mS)$.
\end{proposition}

Remark that, as soon as $|\tan{\theta}-\tan{\theta_r}|>0$, the exponential growth factor of the half-plane 
model becomes greater than that of the quarter-plane model, and Algorithm~\ref{alg:fullmonty} becomes exponential on $n$.
On the other hand, for any length $n\in\mathbb{N}$, setting $\delta:=1/(n+1)$ will define a model $\mH_{\theta_r}(\mS)$ which coincides with $\mH_{\theta}(\mS)$ on positive walks of length $n$. Indeed, the accumulated {\em error} due to the approximation of the step set remains too small to lead to the acceptance of some walk in $\mH_{\theta_r}(\mS)$ and not in $\mH_{\theta}(\mS)$ (and vice-versa).

\begin{theorem}[Complexity of $1/(n+1)$-rational approximation]
Let $\walks (\mathbb{R}_{\geq 0}, \mA, n)$ be a non-trivial unidimensional model defined by a non-rational multiset~$\mA$. The uniform random generation of $k$ walks of length~$n$ in
$\walks (\mathbb{R}_{\geq 0}, \mA, n)$ can be performed in:
\begin{itemize}
\item $\mathcal{O}(k\cdot n\log n + n^3)$ arithmetic operations, using storage for $\mathcal{O}(n^3)$ large integers~\cite{FlZiVa94}; 
\item $\mathcal{O}(k\cdot n^3\log n + n^2)$ arithmetic operations, using storage for $\mathcal{O}(n^2)$ large integers~\cite{Goldwurm1995};
\item $\mathcal{O}(k\cdot n^2 + n^2)$ arithmetic operations, using storage for $\mathcal{O}(n^2)$ real values (Oracle)~\cite{DuFlLoSc02,Pivoteau2008}.
\end{itemize}
\end{theorem}

Finally, we conjecture that a polynomial rejection is 
actually reached using a $(1/\sqrt{n})$-rational approximation.
\TODOYann{To be replaced by a suitable conjecture}
\begin{conjecture}
Let $h_{n}(\theta)$ be the number of walks of length $n$ in an half-plane model $\mH_{\theta}(\mS)$,
there exists an infinite sequence of angles $\{\theta_n\}_{n\le0}$ such that:
\begin{itemize}
\item For all $n\ge0$,  $\mH_{\theta_n}(\mS)$ is a $(1/\sqrt{n})$-rational approximation of $\mH_{\theta^*}(\mS)$;
\item The number of rejections in Algorithm~\ref{alg:fullmonty} remains polynomial in $n$:
\[
\exists p\in \mathbb{R}, \lim_{n\to+\infty} \frac{h_{n}(\theta_n)}{h_{n}(\theta)} \in \mathcal{O}(n^p)
\]
\end{itemize}
\end{conjecture}

\section{Remarks and future extensions}
\label{sec:Results}
We have implemented this algorithm in Python, with external calls to
Sage to compute $\theta^*$, and to Maple for Boltzmann Generation.
Experimentally, in the case of irrational projected steps, even crude
approximations led to much improved empirical complexities than both the 
default half-plane generation and the naive recursive generator. 
On the other hand, increasingly precise approximations led to an overwhelming 
growth in the size of the grammar, as could be expected from the asymptotic complexity. 
This raises interesting questions about the precise interplay
between the size of the grammar and the complexity, starting with our conjecture 
which we hope to address in a future version of this work. 

There are many possible optimizations, notably, anticipated rejection.
We expect this should have a positive effect on the complexity,
particularly in the null-drift cases, possibly after projection onto the 
targeted half-plane.

\paragraph{Natural extensions and generalizations}
Finally, many natural extensions come to mind. Generating excursions in the
quarter plane is difficult, but using our grammar-based approach it is completely 
straightforward. Finally, there are analogous ``best hyperplane'' theorems in
higher dimensions, and for more general cones, and our general approach could in principle 
generalize to these cases.

\subsection*{Acknowledgements}
We are very grateful for discussions with Kilian Raschel and Julien
Courtiel. 
\bibliographystyle{plain}
\bibliography{LMP}
\end{document}